\title{Degree versions of the Erd\H os-Ko-Rado Theorem and Erd\H os hypergraph matching conjecture}
\documentclass[11pt]{article}
\usepackage{amsfonts,amssymb,amsmath,latexsym,color, verbatim}

\newenvironment{proof}
      {\medskip\noindent{\bf Proof.}\hspace{1mm}}
      {\hfill$\Box$\medskip}

\def\qed{\ifvmode\mbox{ }\else\unskip\fi\hskip 1em plus 10fill$\Box$}

\newtheorem{theorem}{Theorem}[section]
\newtheorem{prop}[theorem]{Proposition}
\newtheorem{lemma}[theorem]{Lemma}
\newtheorem{coro}[theorem]{Corollary}
\newtheorem{conj}[theorem]{Conjecture}

\newtheorem{remark}[theorem]{Remark}

\newcommand{\HH}{{\cal H}}

\author{Hao Huang
	\thanks{Department of Math and CS, Emory University, Atlanta, GA 30322. Email: hao.huang@emory.edu.}
	\and Yi Zhao
	\thanks{Department of Mathematics, Georgia State University, Atlanta, GA 30302. Email: yzhao6@gsu.edu. Research supported in part by NSF grant DMS-1400073.}
}

\date{}

\begin{document}
\maketitle
\abstract{We use an algebraic method to prove a degree version of the celebrated Erd\H os-Ko-Rado theorem: given $n>2k$, every intersecting $k$-uniform hypergraph $H$ on $n$ vertices contains a vertex that lies on at most $\binom{n-2}{k-2}$ edges. 
This result can be viewed as a special case of the degree version of a well-known conjecture of Erd\H{o}s on hypergraph matchings. Improving the work of Bollob\'as, Daykin, and Erd\H os from 1976, we show that given integers $n, k, s$ with $n\ge 3k^2 s$, every $k$-uniform hypergraph $H$ on $n$ vertices with minimum vertex degree greater than $\binom{n-1}{k-1}-\binom{n-s}{k-1}$ contains $s$ disjoint edges.
}

\section{Introduction}\label{sec_introduction}

Fix integers $1\le k\le n$ and a set $X$ of $n$ elements.
The celebrated theorem of Erd\H os, Ko and Rado \cite{ekr} states that when $n \ge 2k$, every intersecting family of $k$-subsets on $X$ has at most $\binom{n-1}{k-1}$ members. Moreover, when $n>2k$, the extremal family is unique (up to isomorphism): it consists of all the $k$-subsets of $X$ that contains a fixed element. We call such family a \emph{$1$-star} (more precisely, $(n, k, 1)$-star). The Erd\H os-Ko-Rado theorem is widely regarded as the cornerstone of extremal combinatorics and has many interesting applications and generalizations, see \cite{FrGr89} for different proofs and \cite{DeFr83} for a survey.  One well-known generalization was given by Hilton and Milner \cite{hilton-milner}, who showed that if $n>2k$ and $\HH$ is an intersecting family of $k$-subsets of $X$, then either $|\HH| \le \binom{n-1}{k-1}-\binom{n-k-1}{k-1}+1$, or $\HH$ is a subfamily of a $1$-star. 

We may view a family $\HH$ of $k$-subsets of $X$ as a $k$-uniform hypergraph $H$ with vertex set $X$ and edge set $\HH$. Let $d<k$ be a nonnegative integer. Given a $k$-uniform hypergraph $H$ with a set $S$ of $d$ vertices, the degree of $S$, denoted by $\deg_H(S)$ or simply $\deg(S)$, is the number of edges containing $S$ as a subset. The minimum $d$-degree $\delta_d(H)$ is the minimum of $\deg(S)$ over all the $d$-subsets of $V(H)$. For example, $\delta_0(H)=e(H)$ is the number of edges in $H$, and $\delta_1(H)$ is the minimum vertex degree of $H$. 

In this paper we study $k$-uniform intersecting families from the aspect of the minimum vertex degree. There have been work on intersecting families with maximum degree conditions. For example, Frankl \cite{frankl87} extended the Hilton-Milner theorem by giving sharp upper bounds on the size of intersecting families with certain maximum degree. There were also much recent work on other extremal problems with minimum degree conditions, such as the $d$-degree Tur\'an problems (see, e.g., \cite{MuZh}) and hypergraph Dirac problems (see, e.g., \cite{RRsurvey}).

Our first result is a minimum degree version of the Erd\H os-Ko-Rado theorem. Note that it does not follow from the Erd\H os-Ko-Rado theorem directly because a $k$-uniform hypergraph with $\delta(H)\ge \binom{n-2}{k-2}$ may only have $\binom{n-2}{k-2}\frac{n}k < \binom{n-1}{k-1}$ edges.

\begin{theorem}\label{degree_ekr}
	Given $n \ge 2k+1$, if every vertex in a $k$-uniform hypergraph $H$ on $n$ vertices has degree at least $\binom{n-2}{k-2}$, then either $H$ is a $1$-star, or $H$ contains two disjoint edges.
\end{theorem}

Apparently the minimum degree condition in Theorem~\ref{degree_ekr} is sharp because of $1$-stars. Remark~\ref{remark} shows that the bound $n \ge 2k+1$ is also necessary. Furthermore, it is not difficult to prove Theorem \ref{degree_ekr} for sufficiently large $n$. In fact, let $H$ be an intersecting $k$-uniform hypergraph with $\delta_1(H)\ge \binom{n-2}{k-2}$. Recall that the Hilton-Milner theorem says that when $n>2k$, every  $H$ either satisfies $e(H) \le \binom{n-1}{k-1}-\binom{n-k-1}{k-1}+1$, or is a sub-hypergraph of some $1$-star. 
In the latter case, $H$ must be excactly a $1$-star. In the former case, the minimum degree condition implies that $e(H) \ge \frac{n}{k}\binom{n-2}{k-2}$. Simple computation shows that for $n$ sufficiently large in $k$ (say $n \gg k^2/\log k$), $\frac{n}{k}\binom{n-2}{k-2}>\binom{n-1}{k-1}-\binom{n-k-1}{k-1}+1$ so the former case cannot happen. However, in order to obtain the exact bound $n \ge 2k+1$, we will develop techniques similar to the Hoffman bound, and apply a lemma from combinatorial geometry. 

Our method also allows us to prove a generalization of Theorem \ref{degree_ekr}. We say that two families $\mathcal{B}$ and $\mathcal{C}$ of $k$-subsets of $[n]$ are \textit{cross-intersecting} if for every two sets $B \in \mathcal{B}$ and $C \in \mathcal{C}$, their intersection $B \cap C $ is non-empty. Pyber \cite{pyber_cross} showed that when $n \ge 2k$, every two cross-intersecting families $\mathcal{B}$ and $\mathcal{C}$ satisfy $|\mathcal{B}||\mathcal{C}| \le \binom{n-1}{k-1}^2$. The special case when $\mathcal{B}=\mathcal{C}$ is exactly the Erd\H os-Ko-Rado theorem. We are able to prove a degree version of Pyber's result, which is a generalization of Theorem~\ref{degree_ekr}. 

\begin{theorem}\label{thm_cross}
	For $n \ge 2k+1$, two cross-intersecting families $\mathcal{B}$ and $\mathcal{C}$ of $k$-subsets of $[n]$ satisfy
	$\delta_{1}(\mathcal{B}) \delta_{1}(\mathcal{C}) \le \binom{n-2}{k-2}^2$.
\end{theorem}

A \emph{matching} in a hypergraph $H$ is a collection of vertex-disjoint edges; the \emph{size} of a matching is the number of edges in the matching. The matching number $\nu(H)$ is the maximum size of a matching in $H$. 
The Erd\H os-Ko-Rado theorem, when stated in the language of matchings, says that for $n \ge 2k$, if a $k$-uniform hypergraph $H$ on $n$ vertices has more than $\binom{n-1}{k-1}$ edges, then $\nu(H) \ge 2$. In 1965 Erd\H os \cite{erdos-matching} gave a conjecture that relates the size $e(H)$ of a $k$-uniform hypergraph $H$ to its matching number $\nu(H)$. 

\begin{conj} \cite{erdos-matching}
\label{conj:Erdos}
Every $k$-uniform hypergraph $H$ on $n$ vertices with matching number $\nu(H)<s \le n/k$ satisfies 
$$e(H) \le \max \left\{ \binom{ks-1}{k}, \binom{n}{k}-\binom{n-s+1}{k} \right\}.$$
\end{conj}

For $1\le i\le k$, define $H^i_{n, k, s}$ to be the $k$-uniform hypergraph on $n$ vertices whose vertex set contains a set $S$ of $si-1$ vertices, and whose edge set consists of all $k$-sets $e$ such that $|e\cap S|\ge i$. It is clear that $H^i_{n, k, s}$ contains no matching of size s. Conjecture~\ref{conj:Erdos} says that the largest number of edges in a $k$-uniform hypergraph on $n$ vertices not containing a matching of size $s$ is attained by $H^1_{n, k, s}$ or $H^k_{n, k, s}$.

The $k=2$ case of Conjecture~\ref{conj:Erdos} is a classic result of Erd\H os and Gallai \cite{erdos-gallai}. When $k=3$, Frankl, R\"odl and Ruci\'nski \cite{frankl-rodl-rucinski} proved the conjecture for $s \le n/4$, and {\L}uczak and Mieczkowska \cite{luczak-mieczkowska} proved it for sufficiently large $s$. Recently Frankl \cite{frankl_3-graph} proved the conjecture for $k=3$. For arbitrary $k$, Erd\H{o}s \cite{erdos-matching} proved the conjecture for $n\ge n_0(k, s)$; 
Bollob\'as, Daykin and Erd\H{o}s \cite{bollobas-daykin-erdos} proved the conjecture for $n> 2k^3 (s-1)$. Huang, Loh and Sudakov \cite{huang-loh-sudakov} improved it to $n\ge 3k^2 s$; Frankl, \L uczak, and Mieczkowska \cite{FLM} further improved to $n\ge 2k^2 s/ \log k$. Recently Frankl \cite{frankl-general} proved the conjecture for $n\ge (2s-1)k- s+1$.




Bollob\'as, Daykin, and Erd\H os \cite{bollobas-daykin-erdos} also considered the minimum degree version of Conjecture~\ref{conj:Erdos}. They showed that if $n> 2k^3(s-1)$ and $H$ is a $k$-uniform hypergraph $H$ on $n$ vertices with $\delta(H)> \binom{n-1}{k-1}-\binom{n-s}{k-1}$, then $\nu(H)\ge s$. The construction $H^1_{n,k,s}$ shows that their minimum degree condition is best possible. Our next result improves this result by reducing the bound on $n$ to $n\ge 3k^2 s$.

\begin{theorem}\label{thm:degErdos}
Given $n, k, s$ with $n\ge 3k^2 s$, if the degrees of every vertex in a $k$-uniform hypergraph $H$ on $n$ vertices is strictly greater than $\binom{n-1}{k-1}-\binom{n-s}{k-1}$, then $H$ contains a matching of $s$ disjoint edges.
\end{theorem}

The proof of Theorem~\ref{thm:degErdos} follows the approach used by Huang, Loh and Sudakov \cite{huang-loh-sudakov} and applies Theorem~\ref{degree_ekr} as the base case of the induction.

The rest of the paper is organized as follows. We prove Theorems \ref{degree_ekr} and \ref{thm_cross} in the next section, prove Theorem \ref{thm:degErdos} and discuss its fractional version in Section~\ref{sec_general}.
The final section contains concluding remarks and open problems.

\section{Proofs of Theorems \ref{degree_ekr} and \ref{thm_cross}}
\label{sec_degree-ekr}
As we mentioned earlier in the introduction, it is easy to prove Theorem \ref{degree_ekr} for sufficiently large $n$. To obtain the tight bound $n\ge 2n+1$, our proof uses the spectral method and relies on the following lemma from combinatorial geometry. To understand this lemma intuitively, we would like to point out that the $n=3$ case of this lemma says that given three vectors $\vec{u}_1, \vec{u}_2, \vec{u}_3$ in $\mathbb{R}^2$ such that the angle between any two of the vectors is $120$ degree, then an arbitrary vector $\vec{v}$ must have an angle at least $120$ degree with some of $\vec{u}_i$.

\begin{lemma} \label{lem_simplex}
Suppose $\vec{u}_1, \cdots, \vec{u}_n$ are $n$ distinct non-zero vectors of equal length in $\mathbb{R}^{n-1}$ such that for all $i \ne j$, $\langle \vec{u}_i, \vec{u}_j \rangle$ are equal. Then for an arbitrary vector $\vec{v}$, there exists an index $i$ such that 
$$\langle \vec{v}, \vec{u}_i \rangle \le -\frac{1}{n-1} \langle \vec{u}_i, \vec{u}_i \rangle^{1/2} \langle \vec{v}, \vec{v} \rangle^{1/2}.$$ 
\end{lemma}
\begin{proof}
Without loss of generality, we may assume that all $\vec{u}_i$ 
are unit vectors. Note that for $i \ne j$,  $||\vec{u}_i-\vec{u}_j||^2=||\vec{u}_i||^2+||\vec{u}_j||^2-2\langle \vec{u}_i, \vec{u}_j \rangle=2-2\langle \vec{u}_i, \vec{u}_j \rangle$ is a constant. It is well known that in $\mathbb{R}^{n-1}$, if there are $n$ distinct points on the unit sphere whose pairwise distances are equal, then these points are the vertices of a regular simplex, and their center of mass is at the origin. Therefore we have
$$2\sum_{1 \le i < j \le n} \langle \vec{u}_i, \vec{u}_j \rangle
=  \left\langle \sum_{i=1}^n \vec{u}_i, \sum_{i=1}^n \vec{u}_i \right\rangle
-\sum_{i=1}^n \langle \vec{u}_i, \vec{u}_i \rangle = -n, $$
which implies that for $i \ne j$, $\langle \vec{u}_i, \vec{u}_j \rangle = -n/(n(n-1))= -1/(n-1)$. Denote by $C_i$ the convex cone formed by all the nonnegative linear combinations of the $n-1$ vectors $\vec{u}_1, \cdots, \vec{u}_{i-1}, \vec{u}_{i+1}, \cdots, \vec{u}_n$. It is not hard to see that the union of $C_i$ is equal to $\mathbb{R}^{n-1}$. Therefore an arbitrary vector $\vec{v}$ can be represented as $\sum_{j=1}^n \alpha_j \vec{u}_j$ for some $\alpha_1, \cdots, \alpha_n \ge 0$, such that $\alpha_i=0$ for some index $i$. As a result, 
$$\langle \vec{v}, \vec{u}_i \rangle= \sum_{j=1}^n \alpha_j \langle \vec{u}_j, \vec{u}_i \rangle = -\frac{1}{n-1} \sum_{j=1}^n \alpha_j.$$
On the other hand, since the unit ball in $\mathbb{R}^{n-1}$ is convex and contains $\vec{u}_i$, we have $\vec{v}/(\sum_{j=1}^n \alpha_j)=(\sum_{j=1}^n \alpha_j \vec{u}_j)/(\sum_{j=1}^n \alpha_j)$ is also in the unit ball. Therefore $||v|| \le \sum_{j=1}^n \alpha_j$ and consequently
$$\langle \vec{v}, \vec{u}_i \rangle = -\frac{1}{n-1} \sum_{j=1}^n \alpha_j  \le -\frac{1}{n-1}\langle \vec{v}, \vec{v} \rangle ^{1/2},$$ which concludes the proof.
\end{proof}

Lemma \ref{lem_simplex}, combined with some ideas similar to what was used to prove Hoffman's bound \cite{hoffman}, leads to the proof of Theorem \ref{degree_ekr}.\\

\noindent \textbf{Proof of Theorem \ref{degree_ekr}: }

We start by assuming that the hypergraph $H$ is intersecting and all the vertex degrees are at least $\binom{n-2}{k-2}$. To prove Theorem \ref{degree_ekr}, it suffices to show that under these assumptions $H$ must be a $1$-star. By double counting, the number of edges $e(H)$ is at least $\frac{n}{k} \binom{n-2}{k-2}$.

Let $G$ be the Kneser graph $KG(n, k)$ whose vertices correspond to all the $k$-subsets of a set of $n$ elements, and where two vertices are adjacent if and only if the two corresponding sets are disjoint. Let $A$ be its adjacency matrix. It is known (see for example on Page $200$ of \cite{godsil-royle}) that the eigenvalues of $A$ are $\lambda_j=(-1)^j \binom{n-k-j}{k-j}$ with multiplicity $\binom{n}{j}-\binom{n}{j-1}$ for $j=0, \cdots, k$ (assuming that $\binom{n}{-1}=0$). Moreover, the eigenspace $E_0$ of the eigenvalue $\lambda_0=\binom{n-k}{k}$ is spanned by the unit vector $\vec{v}_1=\vec{1}/(\sqrt{\binom{n}{k}})$. The eigenspace $E_1$ of the eigenvalue $-\binom{n-k-1}{k-1}$ is the orthogonal complement of $E_0$ in the subspace spanned by the characteristic vectors of all the $n$ distinct $1$-stars. We assume that unit vectors $\vec{v}_2, \vec{v}_3, \cdots, \vec{v}_n$ form an orthogonal basis of $E_1$. 
We denote by $\vec{s}_i$ the characteristic vector of the $1$-star centered at $i$. Since $\vec{s}_i$ is contained in $E_0 \bigoplus E_1$, we may further assume that
\begin{align*}
\vec{s}_1=s_{11}\vec{v}_1+ s_{12}\vec{v}_2 + \cdots + s_{1n}\vec{v}_n,\\
\vec{s}_2=s_{21}\vec{v}_1+ s_{22}\vec{v}_2 + \cdots + s_{2n}\vec{v}_n,\\
\vdots~~~~~~~~~~~~~~~~~~~~~~~~\\
\vec{s}_n=s_{n1}\vec{v}_1+ s_{n2}\vec{v}_2 + \cdots + s_{nn}\vec{v}_n.
\end{align*}
Clearly $\vec{v}_1, \cdots, \vec{v}_n$ can be extended to $\vec{v}_1, \cdots, \vec{v}_{\binom{n}{k}}$, which form an orthogonal basis for $\mathbb{R}^{\binom{n}{k}}$, in which $\vec{v}_{{n \choose j-1}+1}, \cdots, \vec{v}_{{n \choose j}}$ are the eigenvectors for $\lambda_j=(-1)^j \binom{n-k-j}{k-j}$.

Now suppose the characteristic vector $\vec{h}$ of the hypergraph $H$ can be represented as $\vec{h}=\sum_{i=1}^{\binom{n}{k}} h_i \vec{v}_i$. Since $H$ is intersecting, we have
\begin{align}
0 &=\vec{h}^T A \vec{h} = \sum_{i=0}^{k} \lambda_i \cdot  \left(\sum_{j=\binom{n}{i-1}+1}^{\binom{n}{i}} h_j^2 \right) \nonumber \\
&=\binom{n-k}{k}h_1^2 - \binom{n-k-1}{k-1} \left(h_2^2+\cdots +h_n^2 \right) +\sum_{i=2}^{k} \lambda_i \cdot  \left(\sum_{j=\binom{n}{i-1}+1}^{\binom{n}{i}} h_j^2 \right). \label{eq:hAh}
\end{align}
Note that $h_1=\langle \vec{h}, \vec{v}_1 \rangle = e(H)/\sqrt{n \choose k}$, and $\sum_{i=1}^{n \choose k} h_i^2=\langle \vec{h}, \vec{h}\rangle=e(H)=\sqrt{n\choose k} h_1$. Moreover, we know that for $i \ge 2$, $\lambda_i =(-1)^i \binom{n-k-i}{k-i} \ge -\binom{n-k-3}{k-3}$. Following from these facts, we have 
\begin{align}\label{initial_hoffman}
0 \ge \binom{n-k}{k}h_1^2 - \binom{n-k-1}{k-1}\left(\sum_{j=2}^{n}h_j^2 \right) -\binom{n-k-3}{k-3}  \left(\sqrt{n \choose k}h_1 - h_1^2 - \sum_{j=2}^n h_j^2 \right).
\end{align}
For $n \ge 2k+1$, since $(n-k-1)(n-k-2) \ge k(k-1) > k(k-2)$, we have
\begin{align*}
\binom{n-k-3}{k-3} 
&=\frac{k(k-2)}{(n-k-1)(n-k-2)} \cdot \frac{k-1}{n-k} \binom{n-k}{k} < \frac{k-1}{n-k}\binom{n-k}{k}.
\end{align*}
Therefore 
\begin{align}\label{hoffman}
0 \ge \binom{n-k}{k}h_1^2 - \binom{n-k-1}{k-1}\left(\sum_{j=2}^{n}h_j^2 \right) -\frac{k-1}{n-k} \binom{n-k}{k}  \left(\sqrt{n \choose k}h_1 - h_1^2 - \sum_{j=2}^n h_j^2 \right).
\end{align} 
Simplifying this inequality, we obtain
\begin{align}\label{hoffman_cor}
\sum_{j=2}^n h_j^2 
\ge (n-1) h_1^2 - (k-1)\sqrt{n \choose k} h_1 =\frac{n-1}{\binom{n}{k}}e(H)\left(e(H)-\frac{n}{k}\binom{n-2}{k-2}\right).
\end{align}

We now show that the vectors $\vec{u}_i=(s_{i2}, s_{i3}, \cdots, s_{in})$, $i=1, \cdots, n$, satisfy the assumptions in Lemma \ref{lem_simplex}. First, we have
\begin{align*}
\langle \vec{u}_i, \vec{u}_i \rangle &=s_{i2}^2+\cdots + s_{in}^2= \left(\sum_{j=1}^n s_{ij}^2 \right)-s_{i1}^2=\langle \vec{s}_i, \vec{s}_i \rangle - 
\langle \vec{s}_i, \vec{v}_1 \rangle^2\\
&=\binom{n-1}{k-1}- \binom{n-1}{k-1}^2/\binom{n}{k} =\frac{(n-k)k}{n^2} \binom{n}{k}
\end{align*}
So all the vectors $\vec{u}_i$ are of equal lengths. On the other hand, we have that for $i \ne j$, 
\begin{align*}
\langle \vec{u}_i, \vec{u}_j \rangle &= \langle \vec{s}_i, \vec{s}_j \rangle -s_{i1}s_{j1}=\langle \vec{s}_i, \vec{s}_j \rangle -\langle \vec{s}_i, \vec{v}_1 \rangle \langle \vec{s}_j, \vec{v}_1 \rangle\\ 
&=\binom{n-2}{k-2}-\binom{n-1}{k-1}^2/\binom{n}{k}=-\frac{k(n-k)}{n^2(n-1)} \binom{n}{k}
\end{align*}
Applying Lemma \ref{lem_simplex} with $\vec{v}=(h_2, \cdots, h_n)$, we find an index $i$, such that  
\begin{align}
\sum_{j=2}^n h_j s_{ij}&=\langle \vec{v}, \vec{u}_i \rangle 
\le -\frac{1}{n-1}\langle \vec{u}_i, \vec{u}_i \rangle^{1/2}\left(\sum_{j=2}^n h_j^2 \right)^{1/2} \nonumber\\
&=-\frac{1}{n-1}\left(\binom{n}{k} \cdot \frac{k(n-k)}{n^2}\right)^{1/2}  \left(\sum_{j=2}^{n}h_j^2 \right)^{1/2}. \label{eq:hs2}
\end{align}

On the other hand degree of the vertex $i$ in $H$ is equal to  $\langle \vec{h}, \vec{s}_i \rangle=\sum_{j=1}^n h_j s_{ij} \ge \binom{n-2}{k-2}$. Since  $h_1= e(H)/ \sqrt{n \choose k}$ and $s_{i1}=\binom{n-1}{k-1}/\sqrt{n \choose k}$, we also have 
\[
h_1 s_{i1} = \frac{e(H)}{\sqrt{\binom nk}} \frac{\binom{n-1}{k-1}}{\sqrt{\binom nk}} = \frac kn e(H).
\]
Consequently $\sum_{j=2}^n h_j s_{ij} \ge \binom{n-2}{k-2} - e(H) k/n$.
Together with \eqref{eq:hs2}, this gives
\[
e(H) \frac kn - \binom{n-2}{k-2} \ge \frac{1}{n-1}\left(\binom{n}{k} \cdot \frac{k(n-k)}{n^2}\right)^{1/2}  \left(\sum_{j=2}^{n}h_j^2 \right)^{1/2},
\]
which implies 
\begin{align}\label{simplex}
\sum_{j=2}^n h_j^2 & \le \left(e(H)-\frac{n}{k}\binom{n-2}{k-2}\right)^2\cdot \frac{(n-1)^2 k}{(n-k)\binom{n}{k}}.
\end{align}
Combining the inequalities \eqref{hoffman_cor} and \eqref{simplex}, we obtain that
\[
\frac{n-1}{\binom{n}{k}}e(H)\left(e(H)-\frac{n}{k}\binom{n-2}{k-2}\right) \le  \left(e(H)-\frac{n}{k}\binom{n-2}{k-2}\right)^2\cdot \frac{(n-1)^2 k}{(n-k)\binom{n}{k}},
\]
equivalently,
\[
\left (e(H)-\frac{n}{k}\binom{n-2}{k-2}\right) \left (e(H)- \binom{n-1}{k-1}\right)\ge 0
\]
Thus either $e(H) \le \frac{n}{k}\binom{n-2}{k-2}$ or $e(H) \ge \binom{n-1}{k-1}$.
In the former case we must have $e(H)=\frac{n}{k} \binom{n-2}{k-2}$ because of the minimum degree condition. From \eqref{hoffman_cor} and \eqref{simplex}, we derive that $\sum_{j=2}^n h_j^2=0$, and equalities are attained in \eqref{hoffman_cor} and \eqref{simplex} . Recall that we used $\binom{n-k-3}{k-3}<\frac{k-1}{n-k} \binom{n-k}{k}$ to get \eqref{hoffman} from \eqref{initial_hoffman}. We must have 
$$\sqrt{n \choose k}h_1 - h_1^2 - \sum_{j=2}^n h_j^2=0,$$
equivalently $h_1 = \sqrt{\binom nk}$ otherwise \eqref{hoffman} becomes a strict inequality. 
It follows that $e(H)=\sqrt{\binom{n}{k}}h_1 = {n \choose k}>\frac{n}{k}\binom{n-2}{k-2}$, a contradiction.

Therefore $e(H) \ge \binom{n-1}{k-1}$. For $n \ge 2k+1$, since the $1$-star is the unique example that attains the maximum in the Erd\H os-Ko-Rado theorem, the intersecting hypergraph $H$ must be the $1$-star, which completes the proof.
\qed

\begin{remark}\label{remark}
We would like to point out that the $n \ge 2k+1$ bound in Theorem \ref{degree_ekr} is sharp. 
Indeed, unlike the normal Erd\H os-Ko-Rado theorem, this degree analogue may no longer be valid for $n=2k$. For example, when $n=6$ and $k=3$, $H=\{123, 234, 345, 451, 512, 136, 246, 356, 256, 146\}$ is an intersecting hypergraph in which every vertex has degree exactly $5$, which is strictly greater than $\binom{n-2}{k-2}=4$. In general, for $n=2k$ and $k$ large, we can partition all the $k$-subsets of $[2k]$ into $\binom{2k}{k}/2$ pairs, and randomly select exactly one $k$-subset from each pair. Such hypergraph is clearly intersecting, and the expected vertex degree is equal to $(\binom{2k}{k}/2)\cdot (k/(2k))=\binom{2k}{k}/4$, which is strictly greater than $\binom{n-2}{k-2}=\frac{k-1}{4k-2}\binom{2k}{k}$. 
\end{remark}

The spectral method, together with some Cauchy-Schwarz type inequalities, allows us to prove Theorem \ref{thm_cross}, which is a generalization of Theorem \ref{degree_ekr} to cross-intersecting families.\\

\noindent \textbf{Proof of Theorem \ref{thm_cross}: }
We assume that $\vec{v}_1, \cdots, \vec{v}_{\binom{n}{k}}$ are the same eigenvectors of the adjacency matrix $A$ of the Kneser graph as in the previous proof, and define $\vec{s}_i$, $s_{ij}$, and $\vec{u}_i$ as before. Let $\vec{b}$ and $\vec{c}$ be the characteristic vectors of $\mathcal{B}$ and $\mathcal{C}$. That $\mathcal{B}$ and $\mathcal{C}$ are cross-intersecting implies that $\vec{b}^T \cdot A \cdot \vec{c}=0$. Suppose $\vec{b}=\sum_i b_i \vec{v}_i$ and $\vec{c}=\sum_i c_i \vec{v}_i$. Let $B=\sum_{j=2}^n b_j^2$ and $C=\sum_{j=2}^n c_j^2$. Similarly as \eqref{eq:hAh}, we have 
\begin{align*}
0 &=\binom{n-k}{k}b_1c_1-\binom{n-k-1}{k-1}\sum_{j=2}^n b_jc_j + \sum_{i=2}^k \lambda_i \cdot \left(\sum_{j=\binom{n}{i-1}+1}^{\binom{n}{i}} b_jc_j \right)\\
&\ge \binom{n-k}{k}b_1c_1-\binom{n-k-1}{k-1} \sqrt{BC} - 
\frac{k-1}{n-k}\binom{n-k}{k} \left(\sum_{j=n+1}^{\binom{n}{k}} b_jc_j \right),
\end{align*}
by using $\lambda_i \ge -\binom{n-k-3}{k-3}> -\frac{k-1}{n-k} \binom{n-k}{k}$ for $i \ge 2$. Applying the Cauchy-Schwarz inequality, we further have
\begin{align}\label{eq:bc1}
0 \ge \binom{n-k}{k}b_1c_1-\binom{n-k-1}{k-1} \sqrt{BC} - 
\frac{k-1}{n-k}\binom{n-k}{k} \left(\sum_{j=n+1}^{\binom{n}{k}} b_j^2\right)^{1/2} \left(\sum_{j=n+1}^{\binom{n}{k}} c_j^2\right)^{1/2}.
\end{align}
Note that $\sum_{j=n+1}^{n\choose k} b_j^2 = |\mathcal{B}|-|\mathcal{B}|^2/\binom{n}{k}-B$ and $\sum_{j=n+1}^{n\choose k} c_j^2 = |\mathcal{C}|-|\mathcal{C}|^2/\binom{n}{k}-C$. Since 
\begin{align} \label{eq:cs}
\sqrt{(x_1-y_1)(x_2-y_2)} \le \sqrt{x_1 x_2}-\sqrt{y_1 y_2}
\end{align}
for all $x_1 \ge y_1 \ge 0$ and $x_2 \ge y_2 \ge 0$, we have
\begin{align}
\left(\sum_{j=n+1}^{n\choose k}b_j^2\right)^{1/2} \left(\sum_{j=n+1}^{n\choose k}c_j^2\right)^{1/2} 
&\le \sqrt{\left( |\mathcal{B}|-|\mathcal{B}|^2/\binom{n}{k} \right) \left(|\mathcal{C}|-|\mathcal{C}|^2/\binom{n}{k}\right) }-\sqrt{BC} \nonumber\\
&\le \sqrt{|\mathcal{B}||\mathcal{C}|}-|\mathcal{B}||\mathcal{C}|/\binom{n}{k}-\sqrt{BC}. \label{eq:bc2}
\end{align}
Using $b_1=|\mathcal{B}|/\sqrt{n \choose k}$ and $c_1=|\mathcal{C}|/\sqrt{n \choose k}$, we can combine \eqref{eq:bc1} and \eqref{eq:bc2} and obtain:
\begin{align}\label{cross_1}
\sqrt{BC} \ge \frac{n-1}{\binom{n}{k}}\sqrt{|\mathcal{B}||\mathcal{C}|}\left(\sqrt{|\mathcal{B}||\mathcal{C}|}-\frac{n}{k}\binom{n-2}{k-2}\right).
\end{align}
Denote by $\beta$ (resp. $\gamma$) the vector formed by taking the second to the $n$-th coordinate of $\vec{b}$ (resp. $\vec{c}$). Recall that $\vec{u}_i$, $1\le i\le n$, satisfy the conditions of Lemma \ref{lem_simplex}. Applying Lemma \ref{lem_simplex} twice, we find indices $i, j$ such that
\[
\langle \vec{\beta}, \vec{u}_i \rangle \le -\frac{1}{n-1}\left(\binom{n}{k}\frac{k(n-k)}{n^2}\right)^{1/2} B^{1/2}, \quad\text{and} \quad
\langle \vec{\gamma}, \vec{u}_j \rangle \le -\frac{1}{n-1}\left(\binom{n}{k}\frac{k(n-k)}{n^2}\right)^{1/2} C^{1/2}.
\]
Since $b_1s_{i1} = |\mathcal{B}| k/n$ and $c_1 s_{j1} = |\mathcal{C}| k/n$, it follows that
\begin{align}
\langle \vec{b}, \vec{s}_i \rangle \langle \vec{c}, \vec{s}_j \rangle &= (b_1s_{i1} + \langle \vec{\beta}, \vec{u}_i \rangle)(c_1s_{i1} + \langle \vec{\gamma}, \vec{u}_i \rangle) \nonumber \\
&\le \left(\frac{k}{n}|\mathcal{B}| -\frac{1}{n-1}\left(\binom{n}{k} \frac{k(n-k)}{n^2}\right)^{1/2} B^{1/2}\right)\left(\frac{k}{n}|\mathcal{C}|-\frac{1}{n-1}\left(\binom{n}{k} \frac{k(n-k)}{n^2}\right)^{1/2} C^{1/2}\right) \nonumber \\
&\le \left(\frac{k}{n}\sqrt{|\mathcal{B}||\mathcal{C}|}-\frac{1}{n-1}\left(\binom{n}{k} \frac{k(n-k)}{n^2}\right)^{1/2} (BC)^{1/4} \right)^2, \label{eq:bscs}
\end{align}
where we apply \eqref{eq:cs} in the last inequality.
Now suppose the conclusion is false, that is, $\delta_1(\mathcal{B})\delta_1(\mathcal{C})>\binom{n-2}{k-2}^2$. 
Then $\langle \vec{b}, \vec{s}_i \rangle \langle \vec{c}, \vec{s}_j \rangle>\binom{n-2}{k-2}^2$.
Together with \eqref{eq:bscs}, this implies that  
\begin{align} \label{cross_2}
\sqrt{BC}< \left(\sqrt{|\mathcal{B}||\mathcal{C}|}-\frac{n}{k}\binom{n-2}{k-2}\right)^2 \cdot \frac{(n-1)^2 k}{(n-k)\binom{n}{k}}.
\end{align}
The inequalities \eqref{cross_1} and \eqref{cross_2} together give 
$$\left(\sqrt{|\mathcal{B}||\mathcal{C}|}-\frac{n}{k}\binom{n-2}{k-2}\right)\left(\sqrt{|\mathcal{B}||\mathcal{C}|}-\binom{n-1}{k-1}\right)>0.$$
Therefore we have $|\mathcal{B}||\mathcal{C}| >\binom{n-1}{k-1}^2$. Pyber's theorem tells us that in this case, the families $\mathcal{B}$ and $\mathcal{C}$ cannot be cross-intersecting, which contradicts the assumption and proves $\delta_1(\mathcal{B})\delta_1(\mathcal{C}) \le \binom{n-2}{k-2}^2$.
\qed

\section{Larger matching number} \label{sec_general}
In this section we prove Theorem~\ref{thm:degErdos} alone with its fractional version. 
Our proof of Theorem~\ref{thm:degErdos} needs the following result from \cite{huang-loh-sudakov}.
\begin{lemma}\cite[Corollary 3.2]{huang-loh-sudakov} \label{rainbow_matching}
If a $k$-uniform hypergraph $H$ on $n$ vertices has $s$ distinct vertices $v_1, \cdots, v_s$ with degrees $\deg(v_i) > 2(s-1)
\binom{n-2}{k-2}$ and $ks<n$, then $H$ contains $s$ disjoint edges. 
\end{lemma}

\medskip
\noindent \textbf{Proof of Theorem~\ref{thm:degErdos}.}
We need to show that every $k$-uniform hypergraph $H$ on $n\ge 3k^2 s$ vertices with $\delta_1(H)> \binom{n-1}{k-1}-\binom{n-s}{k-1}$ contains a matching of size $s$. We prove this by induction on $s$. The $s=2$ case follows from Theorem \ref{degree_ekr} immediately because $n\ge 6k^2 > 2k$. Suppose that assertion holds for $s-1$. Let $H$ be a $k$-uniform hypergraph $H$ with $\delta_1(H)> \binom{n-1}{k-1}-\binom{n-s}{k-1}$. First 
assume $\deg(v) > k(s-1)\binom{n-2}{k-2}$ for some vertex $v$. Let $H_v$ be the sub-hypergraph of $H$ induced on $V(H) \setminus \{v\}$. For every vertex $u \in V(H_v)$, $u$ and $v$ are contained in at most $\binom{n-2}{k-2}$ edges of $H$. Therefore the degree of $u$ in $H_v$ is strictly greater than
\[
\binom{n-1}{k-1}-\binom{n-s}{k-1}-\binom{n-2}{k-2}=\binom{n-2}{k-1}-\binom{n-s}{k-1}.
\]
Since $3k^2\le n/s \le (n-1)/(s-1)$, we can apply the induction hypothesis and obtain a matching of size $s-1$ in 
$H_v$ Each of the $(s-1)k$ vertices in this matching is contained at most $\binom{n-2}{k-2}$ edges with $v$. Since $\deg_H(v)>k(s-1) \binom{n-2}{k-2}$, there exists an edge containing $v$ but disjoint from the $s-1$ edges in the matching -- together they form a matching of size $s$ in $H$.

We may thus assume that the maximum degree of $H$ is at most $k(s-1)\binom{n-2}{k-2}$. 	
Without loss of generality, assume that $\deg(v_1) \ge \deg(v_2) \ge \cdots \ge \deg(v_n)$.
If $\deg(v_s)>2(s-1)\binom{n-2}{k-2}$, then by Lemma \ref{rainbow_matching}, $H$ contains $s$ disjoint edges $e_1, \cdots, e_s$ such that $e_i$ contains $v_i$.  What remains to consider is the case when $\deg(v_1), \cdots, \deg(v_{s-1}) \le k(s-1) \binom{n-2}{k-2}$, and $\deg(v_s), \cdots, \deg(v_n) \le 2(s-1)\binom{n-2}{k-2}$. We first apply the inductive hypothesis and obtain a 
matching $M$ of size $s-1$ in $H$ (one smaller than we would like to prove). The sum of the degrees of the $(s-1)k$ vertices in $M$ is at most 
\[
(s-1) \cdot k(s-1)\binom{n-2}{k-2}+(s-1)(k-1)\cdot 2(s-1)\binom{n-2}{k-2} = (3k-2) (s-1)^2 \binom{n-2}{k-2}.
\] 
On the other hand, $e(H)$, the number of edges of $H$, is greater than $\frac{n}{k}\left(\binom{n-1}{k-1}-\binom{n-s}{k-1}\right) > \frac{n}{k} (s-1) \binom{n-s}{k-2} $.
Note that  
\[
\binom{n-s}{k-2} / \binom{n-2}{k-2} > \left( 1 - \frac{s-2}{n-k} \right)^{k-2} \ge 1 - \frac{(s-2)(k-2)}{n-k} \ge 1 - \frac{1}{3k}
\]
because $n - k\ge 3k (s-2)(k-2)$. Therefore, 
\[
e(H) > \frac{n(s-1)}{k} \left(1- \frac{1}{3k} \right)\binom{n-2}{k-2} > (3k-2) (s-1)^2 \binom{n-2}{k-2}
\]
because $n\ge 3k^2 s$. Since the sum of the degrees of the vertices in $M$ is smaller than $e(H)$, there is an edge of $H$ disjoint from $M$, which together with $M$ gives a matching of size $s$ as desired.
\qed

\bigskip

If we are interested in fractional matchings, then we can prove an analog of Theorem~\ref{thm:degErdos} with a linear bound on $n$ by using the result of Frankl \cite{frankl-general} .
A \textit{fractional matching} in $H$ is a function $w: E(H) \rightarrow [0, 1]$ such that for each $v \in V(H)$, we have $\sum_{v \in e} w(e) \le 1$. The \emph{size} of a fractional matching is $\sum_{e} w(e)$. The size of 
the largest fractional matching of $H$ is called the \emph{fractional matching number} of $H$ and usually denoted by $\nu^*(H)$. Trivially $\nu^*(H) \ge \nu(H)$. 
By the strong duality, $\nu^*(H)$ is equal to the fractional covering number $\tau^*(H)$, which is the minimum of $\sum_{v \in V(H)} f(v)$ over all functions $f: V(H) \rightarrow [0, 1]$ such that for each edge $e \in E(H)$, we have $\sum_{v \in e} f(v) \ge 1$. 

Given integers $n\ge k> d\ge 0$ and a real number $0< s\le n/k$, define $f^s_d(k, n)$ to be the smallest integer $m$ such that every $k$-graph $H$ on $n$ vertices with $\delta_d (H) \geq m$ contains a fractional matching of size $s$.
The following (simple) proposition relates  $f^s_1(k, n)$ to $f_0^s(k-1, n-1)$. Its proof is almost identical to the one of \cite[Proposition 1.1]{afhrrs} -- we include it for completeness.
\begin{prop} \label{lemma_reduction_edge}
For $k\ge 3$ and $s\le n/k$, we have $f^s_1(k, n)\le f_0^s(k-1, n-1)$.
\end{prop}

\begin{proof}
Let $s \le n/k$ and $H$ be a $k$-uniform hypergraph on $H$ on $n$ vertices with $\delta_1(H) \ge f_0^s(k-1, n-1)$.
Since $\nu^*(H)=\tau^*(H)$, it suffices to show that $\tau^*(H) \ge s$. Suppose the vertices of $H$ are indexed by $[n]$, and $\tau^*(H)=\sum_{i=1}^n x_i$, with $\sum_{i \in e} x_i \ge 1$ for all edges $e \in E(H)$. Without loss of generality, we may assume that $x_1 \ge x_2 \ge \cdots \ge x_n \ge 0$. The assumption that 
$\delta_1(H) \ge f_0^s(k-1, n-1)$ means that there are at least $f_0^s(k-1, n-1)$ $(k-1)$-subsets of $\{x_1, \cdots, x_{n-1}\}$ whose sum with $x_n$ is at least $1$.

We consider the hypergraph $H'$ formed by all the $(k-1)$-sets $S$ of $\{1, \cdots, n-1\}$ such that $\left(\sum_{i \in S}x_i \right)+ x_n \ge 1$, then $H'$ is a $(k-1)$-uniform hypergraph on $n-1$ vertices, with $e(H') \ge f_0^s(k-1, n-1)$. By the definition of $f_0^s(k-1, n-1)$, there is a fractional matching of size $s$ in $H'$. Denote by $w: E(H') \rightarrow [0, 1]$ this fractional matchings, then for every $i$, 
$\sum_{i \in e} w(e) \le 1$, and $\sum_{e} w(e) = s$.
Then since for every $i$, $x_i \ge x_n$, and for every edge $e \in E(H')$, $\sum_{j \in e} x_j \ge 1-x_n$, and moreover $n\ge ks$,
\begin{align*}
x_1+\cdots+x_n &=\sum_{i} (\sum_{i \in e} w(e))x_i + \sum_i (1-\sum_{i \in e} w(e))x_i\\
&\ge \sum_e \sum_{i \in e} x_i + (\sum_i (1- \sum_{i \in e} w(e)))x_n\\
&\ge s(1-x_n) + (ks-(k-1)s)x_n=s.
\end{align*}
Therefore we have $\nu^*(H) =\sum_{i=1}^n x_i \ge s$.
\end{proof}

Frankl's \cite{frankl-general} proved Conjecture~\ref{conj:Erdos} for $n \ge (2s-1)k -s+1$, which implies that $f^s_0(k, n)= \binom{n}{k} - \binom{n-s+1}{k} + 1$. By Proposition \ref{lemma_reduction_edge}, when $n-1\ge (2s-1)(k-1)-s+1$, we have
\[
f^s_1(k, n) \le f^s_0(n-1, k-1)= \binom{n-1}{k-1} - \binom{n-s}{k-1} + 1.
\]
We rewrite this as the following corollary.
\begin{coro}
If $n \ge (2s-1)(k-1)-s+2$ and $H$ is a $k$-uniform hypergraph on $n$ vertices with $\delta_1(H) > \binom{n-1}{k-1}-\binom{n-s}{k-1}$, then $\nu^*(H) \ge s$.
\end{coro}

\section{Concluding remarks}
In this paper, we consider the degree versions of the Erd\H os hypergraph matching conjecture and settle the problem for the case when the matching number is one, proving a degree version of the Erd\H os-Ko-Rado theorem. Many intriguing problems with respect to minimum degrees remain open.

\begin{list}{\labelitemi}{\leftmargin=1em}
\item Hilton and Milner \cite{hilton-milner} proved that if $n \ge 2k+1$, and $H$ is an intersecting hypergraph that is \textit{non-trivial}, meaning that the edges of $H$ do not share a common vertex, then $e(H) \le \binom{n-1}{k-1}-\binom{n-k-1}{k-1}+1$. The edge set of the extremal graph $HM_{n,k}$ consists of a $k$-subset $S$, and all the $k$-sets containing a fixed element $v \not \in S$ and intersects with $S$. It is not hard to see that the minimum degree $\delta_1(HM_{n, k})=\binom{n-2}{k-2}-\binom{n-k-2}{k-2}$. The following question naturally arises: is it true that for $n \ge 2k+1$, if $H$ is a non-trivial intersecting hypergraph, then $\delta_1(H) \le \binom{n-2}{k-2}-\binom{n-k-2}{k-2}$? 
In a forthcoming paper, we prove this for sufficiently large $n$.

\item In Theorem \ref{thm_cross}, we show that for two cross-intersecting families $\mathcal{B}$ and $\mathcal{C}$, the product of the minimum degrees $\delta_1(\mathcal{B}) \delta_1(\mathcal{C}) \le \binom{n-2}{k-2}^2$. Could it be true that for sufficiently large $n$, indeed we can have $\min_{i \in [n]} \deg_\mathcal{B}(i) \deg_\mathcal{C}(i) \le \binom{n-2}{k-2}^2$? In other words, if the product of the degrees of the corresponding vertices in $\mathcal{B}$ and $\mathcal{C}$ is large, does this guarantee the existence of disjoint cross pairs?

\item Recall that $H^1_{n, k, s}$ is the $k$-uniform hypergraph on $n$ vertices whose vertex set contains a set $S$ of $s-1$ vertices, and whose edge set consists of all $k$-sets $e$ such that $|e\cap S|\ge 1$. 
Theorem~\ref{thm:degErdos}  shows that when $n\ge 3k^2 s$, the largest minimum vertex degree of a $k$-uniform hypergraph not containing a matching of size $s$ is attained by $H^1_{n, k, s}$. We conjecture the same holds when $n> ks$  (the $n=ks$ case is excluded because of Remark~\ref{remark}).
\begin{conj}
Given positive integers $s, k, n$ such that $s< n/k$, let $H$ be a $k$-uniform hypergraph $H$ on $n$ vertices. If $\delta_1(H)>  \binom{n-1}{k-1}-\binom{n-s}{k-1}$, then $H$ contains a matching of size $s$.
\end{conj}
Our Theorem~\ref{degree_ekr} verifies this conjecture for $s=2$. A result of K\"uhn, Osthus and Treglown \cite{KOT} confirmed the case when $k=3$ and $n$ is sufficiently large.
A more general conjecture on the minimum $d$-degree was given in \cite{zhao_survey}.


%


\end{list}

\end{document}